\documentclass[11pt]{amsart}
\usepackage[latin1]{inputenc}
\usepackage{amsmath}
\usepackage{amsfonts}
\usepackage{amssymb}
\usepackage{graphicx}
\usepackage{fourier}
\usepackage{dsfont}
\usepackage{hyperref}
\usepackage{enumerate}
\usepackage{esint}
\usepackage{bm}
\usepackage{xcolor} 
\usepackage{verbatim}        % for comments
\usepackage{bbm}

\newtheorem{theorem}{Theorem}[section]
\newtheorem{lemma}[theorem]{Lemma}

\theoremstyle{definition}
\newtheorem{definition}[theorem]{Definition}

\theoremstyle{remark}
\newtheorem{remark}[theorem]{Remark}

\newcommand{\norm}[1]{\left\lVert#1\right\rVert}
\newcommand{\abs}[1]{\left\lvert#1\right\rvert}
\newcommand{\pa}[1]{\left( #1 \right)}
\newcommand{\rpa}[1]{\left[ #1 \right]}
\newcommand{\br}[1]{\left\lbrace #1\right\rbrace}
\newcommand{\R}{\mathbb{R}}

\newcommand{\CC}{\mathbf{C}}
\newcommand{\KK}{\mathbf{K}}
\newcommand{\DD}{\mathbf{D}}
\newcommand{\II}{{\mathbf{I}}}
\newcommand{\K}{\mathcal{K}}
\newcommand{\HH}{\mathcal{H}}
\newcommand{\A}{\mathcal{A}}
\newcommand{\B}{\mathcal{B}}
\numberwithin{equation}{section}

\def\1{\mathbbm{1}}

\definecolor{byzantium}{rgb}{0.44, 0.16, 0.39}

\renewcommand{\div}{\operatorname{div}}

\renewcommand{\Re}{\operatorname{Re}}
\renewcommand{\det}{\operatorname{det}}

\title{Convergence to equilibrium for a degenerate McKean-Vlasov Equation}

\author{Manh Hong Duong$^1$ \& Amit Einav$^2$}
\address{$^1$University of Birmingham, School of Mathematics, Edgbaston, Birmingham,
	B15 2TT, United Kingdom}
\email{h.duong@bham.ac.uk}

\address{$^2$Durham University, School of Mathematical Sciences, Upper Mountjoy Campus, Stockton Road, DH1 3LE, Durham, United Kingdom}
\email{amit.einav@durham.ac.uk}

\thanks{The research of MHD was supported by EPSRC Grants EP/V038516/1 and  EP/W008041/1.}

\begin{document}

\maketitle

\begin{abstract}
 In this work we study the convergence to equilibrium for a (potentially) degenerate nonlinear and nonlocal McKean-Vlasov equation. We show that the solution to this equation is related to the solution of a linear degenerate and/or defective Fokker-Planck equation and employ recent sharp convergence results to obtain an easily computable (and many times sharp) rates of convergence to equilibrium for the equation in question.
\end{abstract}

%\subsection{ToDos}
%\begin{itemize}
%	\item Add Remark concerning rigorous differentiation of integrals
%	\item upper/lower contractivity
%	\item Reformatting of equations?
%\end{itemize}	
	
	\tableofcontents

\section{Introduction}\label{sec:intro}
\subsection{Background}\label{subsec:background}
%The McKean-Vlasov equation, originally proposed by McKean in his 1966 work \cite{McKean1966} where he explored connections between a wider class of Markov processes and non-linear parabolic equations, has been shown to have a strong connection to systems of interacting particle/agents by means of the so-called mean field limit approach. \hd{it seems more natural to place this paragraph after we introduce the equation (1.1)and before the sentence "It can be shown that....."}
 
In this work we consider the following McKean-Vlasov type equation
\begin{equation}
	\label{eq:MKE}
	\partial_t \rho(x,t)=\div\rpa{\CC x \rho(x,t)+ \pa{\int_{\R^d} \KK\cdot (x-y)\rho(y,t)dy}\rho(x,t)+ \DD\nabla \rho(x,t)}, 
\end{equation}
where $t\in (0,\infty)$, $x\in \R^d$, and  $\CC, \KK, \DD$ are $d\times d$ constant matrices that correspond to the drift, interaction and diffusion phenomena of the system, respectively.

Equation \eqref{eq:MKE} is nonlinear and nonlocal in a typical way of McKean-Vlasov type of equations - that is,  via an interaction term which is presented as a convolution with a kernel given by $k(x,y)=\KK\cdot (x-y)$.  For simplicity, we will write $\KK x$ instead of $\KK \cdot x$ from this point onwards.

 McKean-Vlasov PDEs, originally proposed by McKean in his 1966 work \cite{McKean1966} where he explored connections between a wider class of Markov processes and non-linear parabolic equations, have been shown to have a strong connection to systems of interacting particle/agents by means of the so-called mean field limit approach. In particular, it can be shown that \eqref{eq:MKE} arises naturally as the mean-field limit of the following interacting particle system
\begin{equation}
	\label{eq:particle_systems}
	dX_i(t)=-\CC X_i(t)\,dt-\frac{1}{N} \sum_{j=1}^N \KK(X_i(t)-X_j(t))\,dt +\sqrt{2\DD}\,dW_i(t)
\end{equation}
with $i=1,\dots, N$ and where $\br{W_i(t)}_{i=1,\dots,N}$ are $d$-dimensional standard independent Wiener processes. Indeed, as can be seen in the classical works \cite{McKean1966,McKean1967,sznitman1991}, the particle system described by \eqref{eq:particle_systems} satisfies the \textit{propagation of chaos} property, i.e. as $N$ goes to infinity any fixed number of particles in the system become independent, and the limiting SDE equation that describes an average particle in \eqref{eq:particle_systems} is given by
\begin{equation}
	\label{eq:limiting_SDE}
	d\bar{X}(t)=-\CC \bar{X}(t)\,d-\KK(\bar{X}(t)-\mathbb{E}(\bar{X}(t))\,dt +\sqrt{2\DD}dW(t). 
\end{equation}
The law of this limiting particle is a weak solution to our McKean-Vlasov equation \eqref{eq:MKE}. 

Systems such as \eqref{eq:particle_systems}, describing the evolution of $N$ interacting particles where every particle moves under the influence of an external force, an interacting force (Curie-Weiss type interaction) and a stochastic noise, have many applications in physics, biology and social sciences. Such systems have been investigated intensively in the literature and their study remains a very active field. 
In particular, the McKean-Vlasov equation \eqref{eq:MKE} has been considered recently in \cite{duong2018,gomes2018} in the context of interacting diffusions driven by (non-Markovian) coloured noise.  For additional information on the topic we refer the reader to recent review works \cite{golse2016,jabin2017}. 

The goal of our presented work is to consider degenerate and defective McKean-Vlasov type of equations and utilise recent study of similar Fokker-Planck equation to conclude easily computable (and many times sharp) rates of convergence to equilibrium for \eqref{eq:MKE}.

\subsection{The setting of the problem}\label{subsec:setting}
We would like to start and emphasise that our main focus in this paper is the long time behaviour of the solution to \eqref{eq:MKE} and \textit{not} its existence, uniqueness and regularity. The well-posedness and smoothness of the solution to general McKean-Vlasov equations under certain restricted H{\"o}rmander conditions (which are slightly stronger than the assumptions of this work),  has been shown in \cite{antonelli2002} in the one dimensional case. For additional information and settings we refer the interested reader to \cite{Buckdahn2017,carrillo2003,crisan2018}. We will assume from this point onwards that the solution to our equation exists, is unique, and is smooth with finite first moment. We will, in fact, find an explicit and smooth solution to the problem under relatively mild conditions on the initial datum. Thus, in our setting, the uniqueness of solutions to \eqref{eq:MKE} will be enough (see Remark \ref{rem:a_solution=unique_solution} in \S\ref{sec:solving}).

The study of the long time behaviour for \eqref{eq:MKE} which we will present here is heavily motivated by (and relies on) recent works on the Fokker-Planck equation (which in our setting can be viewed as \eqref{eq:MKE} with no interaction term, i.e. $\KK=0$) such as \cite{arnold2014,MoH19} and \cite{arnold2018,arnold2022}. Due to the particular structure and type of McKean-Vlasov equation we consider here, we are able to allow for more degeneracies and find an explicit rate of convergence to the equilibrium of the system under our underline ``distance'' - the Boltzmann entropy. This will not only provide a different approach to this study than those found in recent work on more general (over-damped) McKean-Vlasov equations and Vlasov-Fokker-Planck equations (such as\cite{carrillo2003,guillin2021,Guillin2022,ren2021}), but will also explicitly show the impact of $\CC$, $\KK$, and $\DD$ on the convergence rate.

Following on the above, we present some important notions from the recent studies of Fokker-Planck equations which will allow us to state our main theorem.

\begin{definition}\label{def:entropy}
	The (relative) Boltzmann entropy of a non-negative measurable function $f:\R^d\to \R_+$ with respect to a positive measurable function $g:\R^d\to \R$, both of  which of unit mass, i.e.
	$$\int_{\R^d}f(x)dx=\int_{\R^d}g(x)dx=1,$$
	 is defined as
	$$\HH\pa{f|g} = \int_{\R^d}f(x)\log\pa{\frac{f(x)}{g(x)}}dx.$$
\end{definition}

\begin{remark}\label{rem:about_boltzmann_entropy}
	The Boltzmann entropy plays a fundamental role in many of the studies of convergence to equilibrium for various physically, biologically, and chemically motivated equations. While it is nonlinear and not a distance in nature, the Csisz{\'a}r-Kullback-Pinsker inequality states that for any unit mass $f$ and $g$ as above we have that
	\begin{equation}\nonumber
		\norm{f-g}_{L^1\pa{\R^d}} \leq \sqrt{2\HH\pa{f|g}}
	\end{equation}
	which reassures us that the entropy is enough to give us ``reasonable'' convergence.
	The requirement that both $f $ and $g$ will be of unit mass is very natural in numerous equations which describe the evolution of a probability density of a physical, biological, or chemical phenomenon. The McKean-Vlasov equation is slightly different due to the (quadratic) interaction term but, as we will mention shortly, we are always able to assume that our solution has a unit mass by scaling the interaction matrix $\KK$ and using the fact that mass is conserved under \eqref{eq:MKE} (see Remark \ref{rem:about_the_mass}).
\end{remark}

\begin{definition}\label{def:matrices_conditions}
	Let $\mathcal{A}$ and $\mathcal{B}$ be $d\times d$ matrices. We say that the pair $\pa{\mathcal{A},\mathcal{B}}$ is an admissible pair, or admissible in short, if:
	\begin{enumerate}[(A)]
			\item\label{item:cond_semipositive} $\mathcal{B}$ is positive semi-definite with
		$$1\leq \operatorname{rank}\pa{\mathcal{B}} \leq d.$$
		\item\label{item:cond_positive_stability} $\mathcal{A}$ is positively stable, i.e. all its eigenvalues have positive real part.
		\item\label{item:cond_no_invariant_subspace_to_kernel} There exists no non-trivial $\mathcal{A}^T$-invariant subspace of $\operatorname{ker}\pa{\mathcal{B}}$.
	\end{enumerate}
\end{definition}
It is worth to mention that condition \eqref{item:cond_no_invariant_subspace_to_kernel} in the above definition has been shown to be equivalent to H{\"o}rmander condition for the following operator \cite{arnold2014}
$$L_{\mathcal{A}, \mathcal{B}}f=\mathrm{div}(\mathcal{A} x f+\mathcal{B}\nabla f),$$
associated to the Fokker-Planck equation
$$\partial_t f(x,t)=\mathrm{div}(\mathcal{A} x f+\mathcal{B}\nabla f).$$
Condition \eqref{item:cond_no_invariant_subspace_to_kernel} guarantees that the dynamics can't be trapped in a direction where the diffusion is degenerate as the drift will push it away from this direction

In the subsequent analysis, we will relate the solution of the McKean-Vlasov equation  \eqref{eq:MKE} to that of the above Fokker-Planck equation with $\mathcal{A}$ being a linear combination of $\CC$ and $\KK$ and $\mathcal{B}=\DD$.

Lastly, we introduce a new notion for matrices that will prove to be very useful in our current study
\begin{definition}\label{def:almost_stable}
	We say that a square matrix $\mathcal{A}$ is \textit{almost-}positively stable if each of its eigenvalue satisfies
	\begin{enumerate}[(i)]
		\item Its real part is positive, or
		\item It is zero. In this case, the algebraic and geometric multiplicity of the zero eigenvalue are equal. 
	\end{enumerate}
\end{definition}
%It is almost immediate to see that if $\mathcal{A}$ is almost-positively stable then for any $x\in \R^n$
%\begin{equation}\label{eq:convergence_almost_positively}
%	\lim_{t\to\infty}e^{-\mathcal{A}t}x = P_{\ker \mathcal{A}}x.
%\end{equation}

\begin{remark}\label{rem:on_projection}
	From the definition it is simple to see that if $\A$ is a $d\times d$ almost-positively stable matrix then its associated Jordan matrix, $\bm{J}$, must be of the form
	\begin{equation}\label{eq:jordan_for_almost_positively}
		\bm{J}=\begin{pmatrix}
			\bm{0}_{r\times r} & \bm{0}_{r\times (d-r)} \\
			\bm{0}_{(d-r)\times r} & \bm{J}^\prime_{(d-r)\times (d-r)}
		\end{pmatrix}
	\end{equation}	
	where $\bm{J}^\prime$ is, in general, a Jordan matrix with no zero entries in its diagonal. Consequently, denoting by $\br{v_1,\dots,v_d}$ the basis of generalised eigenvectors of $\A$ such that $\ker \A = \mathrm{span}\br{v_1,\dots,v_r}$ (which can be inferred from $\bm{J}$ and the appropriate similarity matrix) we see that the linear operator 
	\begin{equation}\label{eq:proj}
		P_{\ker\A}\pa{\sum_{i=1}^d \alpha_i v_i} = \sum_{i=1}^r \alpha_i v_i
	\end{equation}
	defines a projection on $\ker\A$\footnote{Indeed, $P_{\ker\A}^2=P_{\ker\A}$ and $P_{\ker\A}\vert_{\ker\A}=\II_{\ker\A}$.}. Moreover, the space 
	\begin{equation}\label{eq:def_of_V}
		V:=\mathrm{span}\pa{\II-P_{\ker\A}} = \mathrm{span}\br{v_{r+1},\dots,v_d}
	\end{equation}
	is an $\A-$ invariant subspace and the Jordan form associated to $\A\vert_{V}$ is given by $\bm{J}^\prime$. We can conclude from this that $A\vert_{V}$ is, in fact, positively stable. 
\end{remark}
With these definitions and remark at hand, we are ready to state our main theorem.

%Next we follow ideas and results presented in \cite{arnold2014} and \cite{MoH19} for Fokker-Planck equations of the form 
%\begin{equation}\nonumber
%	\partial_t \rho(x,t)=\div\rpa{\CC x \rho(x,t)+ \pa{\int_{\R^d} \KK(x-y)\rho(y,t)dy}\rho(x,t)+ \DD\nabla \rho(x,t)}, 
%\end{equation}

\begin{theorem}\label{thm:main}
		Consider the McKean-Vlasov equation \eqref{eq:MKE} with drift, interaction, and diffusion matrices $\CC$, $\KK$, and $\DD$ respectively. Assume that the pair $\pa{\CC+\KK, \DD}$ is admissible and that $\CC$ is almost-positively stable.
	
	Let $\rho(t)$ be the solution to the McKean-Vlasov equation with a unit mass initial datum $\rho_0\in L^1_+\pa{\R^d}$\footnote{$L^1_+\pa{\R^d}$ is the space of all non-negative $L^1\pa{\R^d}$ functions.} that has a finite first moment and for which the existence of a unique smooth solution to the problem is guaranteed. Denote by 
	\begin{equation}\label{eq:def_of_m_1}
		\bm{m}_1 := \int_{\R^d}x\rho_0(x)dx,
	\end{equation}
	\begin{equation}\label{eq:def_of_shift}
		s(t):=\pa{e^{-\CC t}-e^{-\pa{\CC+\KK}t}}\bm{m}_1
		%\int_0^t e^{\pa{\CC+\KK}(s-t)}\KK e^{-\CC s}\bm{m}_1ds,
	\end{equation}
	and
	\begin{equation}\label{eq:equilibrium_state}
		\rho_\infty(x):=\frac{1}{\pa{2\pi \det \K}^{\frac{d}{2}}}e^{-\frac{1}{2}x^T\K x},
	\end{equation}
	with $\K$ being the  unique positive definite solution to the continuous Lyapunov equation
	\begin{equation}\nonumber
		2\DD=\pa{\CC+\KK} \K+\K \pa{\CC+\KK}^T.
	\end{equation} 
	Let
	\begin{equation}\label{eq:def_of_mu}
		\mu: = \min\br{\Re \lambda\;|\; \lambda\text{ is an eigenvalue of }\CC+\KK},
	\end{equation}
	and
	\begin{equation}\label{eq:def_of_nu}
		\nu:=\begin{cases}
			\min\br{\Re \lambda\;|\; \lambda\not=0\text{ is an eigenvalue of }\CC}, & \CC\not= 0,\\
			0, & \CC=0.
		\end{cases}
	\end{equation}
	Then $s(t)$ converges to $s_\infty:=P_{\ker \CC}\bm{m}_1$ and there exists an explicit constant $c>0$ that depends only on $\CC$, $\KK$, and $\DD$, such that
	\begin{equation}\label{eq:explicit_convergence}
		\begin{aligned}
			&\HH\pa{\rho(t)|\rho_\infty\pa{\cdot-s_{\infty}}} \leq c \HH\pa{\rho_0|\rho_\infty}\pa{1+t^{2n_1}}e^{-2\mu t}\\
			&+c\rpa{\pa{1+t^{n_2}}e^{-\nu t}\abs{\pa{\II-P_{\ker \CC}}\bm{m}_1}+\pa{1+t^{n_1}}e^{-\mu t}\abs{\bm{m}_1}}\\
			&\Big[\pa{1+\pa{1+t^{n_1}}e^{-\mu t}}\abs{P_{\ker \CC}\bm{m}_1} + \pa{1+t^{n_2}}e^{-\nu t}\abs{\pa{\II-P_{\ker\CC}}\bm{m}_1}\Big]
		\end{aligned}
		%		\begin{aligned}
			%			\HH&\pa{\rho(t)|f_\infty\pa{\cdot-s_\infty}} \leq c\min\Big\{\pa{1+t^{2n_1}}e^{-2\mu t}\HH\pa{\rho_0|f_\infty}\\
			%			&,\abs{s(t)-s_\infty}\pa{2\abs{s_\infty}+\abs{s(t)-s_\infty}},\abs{s(t)-s_\infty}\Bigg[\abs{P_{\ker \CC}\bm{m}_1} \\
			%			& +\abs{\pa{I-P_{\ker \CC}}\bm{m}_1}\pa{1+t^{2n_2}}e^{-2\nu t}\Big]\Big\}. 
			%		\end{aligned}
	\end{equation}
	where $P_{\ker \CC}$ is the projection on the kernel of $\CC$, and $n_1$ and $n_2$ are the maximal defect of all eigenvalues associated to $\mu$ and $\nu$ respectively\footnote{ Recall that the defect of an eigenvalue is the difference between their algebraic and geometric multiplicity.}.
\end{theorem}

The result above here can be compared to the work of Villani \cite{villani98} where the author has shown that under certain conditions the isotropic and spatially homogeneous Landu equation for Maxwellian molecules can be recast as a Fokker-Planck equation.

\begin{remark}\label{rem:about_the_mass}
	The keen reader would notice that we have required a unit mass initial datum $\rho_0$. It is immediate to see that under the assumption of the existence of a smooth solution, $\rho(t)$, the mass is a conserved quantity\footnote{Indeed
	\begin{equation}\nonumber
			\begin{gathered}
					\frac{d}{dt}m_0(t) 
					=\int_{\R^d} \div\rpa{\CC x \rho(x,t)+ \pa{\int_{\R^d} \KK(x-y)\rho(y,t)dy}\rho(x,t)+ \DD\nabla \rho(x,t)}dx=0.
				\end{gathered}
		\end{equation}}, i.e.
	$$m_0(t):=\int_{\R^d}\rho(x,t)dx = \int_{\R^d}\rho_0(x)dx=m_0.$$
	The McKean-Vlasov equation, \eqref{eq:MKE}, is not invariant under mass scaling. This is due to the interaction term. Replacing $\rho(t)$ by $\widetilde{\rho}(t) = \rho(t)/m_0$, however, yields a \textit{unit mass} solution to the McKean-Vlasov equation 
	\begin{equation}\nonumber
		\partial_t \widetilde{\rho}(x,t)=\div\rpa{\CC x \widetilde{\rho}(x,t)+ \pa{\int_{\R^d} \pa{m_0\KK}(x-y)\widetilde{\rho}(y,t)dy}\widetilde{\rho}(x,t)+ \DD\nabla \widetilde{\rho}(x,t)}, 
	\end{equation}
	which is governed by the matrices $\CC$, $m_0\KK$, and $\DD$. We can thus apply our main theorem in this case as long as the associated matrices satisfy the required conditions. 
	%\footnote{Note that $m_0f_\infty(x)$ would be the equilibrium state for this case and 
	%	$$\HH\pa{\rho(t),m_0f_\infty} = m_0 \HH\pa{\widetilde{\rho}(t), f_\infty}.$$.}.
\end{remark}
\begin{remark}\label{rem:ugly_expression}
	The reader might notice that our expression for the rate of convergence of the solution under the Boltzmann entropy functional, inequality \eqref{eq:explicit_convergence}, seems quite cumbersome. The reason we kept it as such and didn't estimate it more crudely is to not only get a better explicit bound, but also to showcase the impact of the drift matrix $\CC$ on the result. In particular, if $\ker \CC = \br{0}$ we get a much improved rate of convergence and we find that $s_\infty=P_{\ker\CC}\bm{m}_1=0$, i.e. we find that the equilibrium to \eqref{eq:MKE} is the same as that for the Fokker-Planck equation with drift matrix $(\CC+\KK)$ and diffusion matrix $\DD$.
\end{remark}

\subsection{Organisation of the paper}\label{subsec:organisation}
We start the body of our work finding an explicit solution to our equation \eqref{eq:MKE} in \S\ref{sec:solving}, relying on known results from the study of Fokker-Planck equations. \S\ref{sec:convergence} will be dedicated to showing the main result of our paper followed by a short discussion about our results and future research in \S\ref{sec:final}.

\section{Solving the McKean-Vlasov equation}\label{sec:solving}

The goal of this section is to provide an explicit solution to our McKean-Vlasov equation \eqref{eq:MKE}. 
%We remind the reader that the main focus of our work is on the long time behaviour of said solution and as such we will not justify regularity assumption we shall make such as interchanging integration with differentiation and considering moments of the solution (see Remark in the end of his section).

Our main result for this section is:
\begin{theorem}\label{thm:solution_to_MKE}
	Consider the McKean-Vlasov equation \eqref{eq:MKE} with drift, interaction, and diffusion matrices $\CC$, $\KK$, and $\DD$ respectively. Assume that the pair $\pa{\CC+\KK,\DD}$ satisfies condition \eqref{item:cond_no_invariant_subspace_to_kernel} of Definition \ref{def:matrices_conditions}, i.e. that there is no $\pa{\CC+\KK}^T$-invariant subspace of $\ker \DD$. Then, a unit mass solution to the equation with initial datum $\rho_0\in L^1_+\pa{\R^d}\cap L^p\pa{\R^d}$, where $p\in (1,\infty]$, is given by
	\begin{equation}\label{eq:solution_MKE}
		\rho(x,t) = f^{\CC+\KK,\DD}_{FP}\pa{x-s(t),t},
	\end{equation}
	with $f^{\mathcal{A},\mathcal{B}}_{FP}\pa{x,t}$ being the (unique) smooth solution to the Fokker-Planck equation  
	\begin{equation}
	\label{eq: FPE}
	\partial_t f(x,t)=\div\rpa{\mathcal{A}x f(x,t)+  \mathcal{B}\nabla f(x,t)}
%		\begin{cases}
%			\partial_t f(x,t)=\div\rpa{\mathcal{A}x f(x,t)+  \mathcal{B}\nabla f(x,t)}, & x\in \R^d,\;t>0,\\
%			f(x,0)=\rho_0(x),& x\in \R^d,
%		\end{cases}
	\end{equation}
	where $t\in (0,\infty)$, $x\in \R^d$, $f(x,0)=\rho_0(x)$, and $s(t)$ is defined as \eqref{eq:def_of_shift}. 
\end{theorem}
\begin{remark}\label{rem:a_solution=unique_solution}
	Theorem \ref{thm:solution_to_MKE} will guarantee that we have an explicit expression to our solution to \eqref{eq:MKE} as long as we know that $\rho_0\in L^1_+\pa{\R^d}\cap L^p\pa{\R^d}$ and that the solution is unique. 
\end{remark}

\begin{remark}\label{rem:mass_again}
	Much like we've noted in Remark \ref{rem:about_the_mass}, the above theorem can easily be modified for the case we are considering a general solution to the problem by replacing the matrix $\KK$ with $m_0\KK$, where $m_0$ is the mass of the solution. 
\end{remark}

\begin{proof}[Proof of Theorem \ref{thm:solution_to_MKE}]
	To emphasise the point that the mass doesn't matter, we will provide the general proof of this theorem, i.e. we assume that $m_0>0$ is arbitrary.
	
	We start by assuming that our solution is smooth and decays fast enough and rewrite \eqref{eq:MKE} as
	\begin{equation}\label{eq:MKE_simplified_slightly}
		\begin{aligned}
				\partial_t \rho(x,t)&=\div\Big[\CC x \rho(x,t)+ \KK x\underbrace{\pa{\int_{\R^d}\rho(y,t)dy}}_{=m_0}\rho(x,t) \\
				&- \KK\pa{\int_{\R^d} y\rho(y,t)dy}\rho(x,t)+ \DD\nabla \rho(x,t)\Big].
		\end{aligned}
	\end{equation}
	This motivates us to explore the momentum of our solution 
%	and explore the evolution of mass
%	$$m_0(t) = \int_{\R^d}\rho(x,t)dx,$$
%	and momentum
	$$\bm{m}_1(t) := \int_{\R^d} x\rho(x,t)dx.$$
%	Integrating \eqref{eq:MKE} over $\R^d$ and using our assumption of decay we find that
%	\begin{equation}\nonumber
%		\begin{gathered}
%			\frac{d}{dt}m_0(t) = \int_{\R^d}\partial_t \rho\pa{x,t}dx \\
%			=\int_{\R^d} \div\rpa{Cx \rho(x,t)+ \pa{\int_{\R^d} K(x-y)\rho(y,t)dy}\rho(x,t)+ D\nabla \rho(x,t)}dx=0,
%		\end{gathered}
%		%	\begin{aligned}
%			%		&\frac{d}{dt}m_0(t) = \int_{\R^d}\partial_t \rho\pa{x,t}dx \\
%			%		&=\int_{\R^d} \div\rpa{Cx \rho(x,t)+ Kx\pa{\int_{\R^d}\rho(y,t)dy}\rho(x,t) - \pa{\int_{\R^d} y\rho(y,t)dy}\rho(x,t)+ D\nabla \rho(x,t)}dx\\
%			%		&=0.
%			%	\end{aligned}
%	\end{equation}
%	which shows that $m_0(t)=m_0(0)$, which we will denote by $m_0$ from this point onwards. 
	
%	Turning our attention to $\bm{m}_1(t)$, and u
	Using \eqref{eq:MKE_simplified_slightly} together with the conservation of mass, we see that for any $i=1,\dots,d$
	$$\frac{d}{dt}m_{1,i}(t) = \int_{\R^d}x_i \partial_t \rho(x,t)dx $$
	$$=\int_{\R^d}x_i\div\rpa{\pa{\CC+m_0\KK}x \rho(x,t)- \pa{\int_{\R^d}\pa{\KK y}\rho(y,t)dy}\rho(x,t)+ \DD\nabla \rho(x,t)}dx$$
	$$= - \int_{\R^d} \pa{\pa{\CC+m_0\KK}x}_{i} \rho(x,t)dx +  \pa{\int_{\R^d}\pa{\KK y}_i\rho(y,t)dy}\underbrace{\int_{\R^d}\rho(x,t)dx}_{=m_0} - \sum_{j=1}^d d_{i,j} \int_{\R^d}\partial_{x_j}\rho(x,t)dx $$
	$$= - \int_{\R^d} \pa{\CC x}_{i} \rho(x,t)dx=-\sum_{j=1}^dc_{i,j}\int_{\R^d}x_j\rho(x,t)dx = -\pa{\CC\bm{m}_1(t)}_i.$$
	We conclude that 
	$$\frac{d}{dt}\bm{m}_1(t) = -\CC\bm{m}_1(t)$$
	and consequently,
	\begin{equation}\label{eq:evolution_of_first_moment}
		\bm{m}_1(t) = e^{-\CC t}\bm{m}_1(0):=e^{-\CC t}\bm{m}_1.
	\end{equation}
	The above allows us to rewrite \eqref{eq:MKE_simplified_slightly}, and as a result $\eqref{eq:MKE}$, as
	\begin{equation}\label{eq:MKE_pre_final_version}
		\partial_t \rho(x,t)=\div\rpa{\pa{\CC+ m_0\KK}x\rho(x,t) - \KK e^{-\CC t}\bm{m}_1\rho(x,t)+ \DD\nabla \rho(x,t)}.
	\end{equation}
	which could be recast as a Fokker-Planck equation with a ``transport part''
	\begin{equation}\label{eq:MKE_final_version}
		\partial_t \rho(x,t)+\KK e^{-\CC t}\bm{m}_1\cdot\nabla \rho(x,t)=\div\rpa{\pa{\CC+ m_0\KK}x\rho(x,t) + \DD\nabla \rho(x,t)}.
	\end{equation}
	Defining 
	$$\widetilde{\rho}(x,t): = \rho(x+s(t),t)$$
	for a function $s(t)$ which we will soon find, and denoting $y(x,t)=x+s(t)$ we find that
	$$\partial_t \widetilde{\rho}(x,t) = \partial_t \rho(y(x,t),t) + \nabla_y \rho (y(x,t),t)\cdot \frac{ds}{dt}(t)$$
	$$=\partial_t \rho(y(x,t),t) + \KK e^{-\CC t}\bm{m}_1\cdot \nabla_y \rho (y(x,t),t) +\pa{\frac{ds}{dt}(t)-\KK e^{-\CC t}\bm{m}_1}\nabla_y \rho (y(x,t),t)$$
	$$=\div_y\rpa{\pa{\CC+ m_0\KK}y\pa{x,t}\rho(y(x,t),t) + \DD\nabla_y \rho(y(x,t),t)}+\pa{\frac{ds}{dt}(t)-\KK e^{-\CC t}\bm{m}_1}\cdot\nabla_y \rho (y(x,t),t)$$
	$$=\div_x\rpa{\pa{\CC+ m_0\KK}x\widetilde{\rho}(x,t)+ \DD\nabla_x\widetilde{\rho}(x,t)}+\pa{\frac{ds}{dt}(t)-\KK e^{-\CC t}\bm{m}_1+\pa{\CC+m_0\KK}s(t)}\cdot\nabla_y \rho (y(x,t),t).$$
	Thus, if 
	\begin{equation}\label{eq:shift_eq}
		\frac{ds}{dt}(t)=\KK e^{-\CC t}\bm{m}_1-\pa{\CC+m_0\KK}s(t)
	\end{equation}
	we conclude that $\widetilde{\rho}(x,t)$ solved the Fokker-Planck equation \eqref{eq: FPE} with drift matrix $\mathcal{A}=\CC+m_0\KK$ and diffusion matrix $\mathcal{B}=\DD$.
	
	It is easy to verify that a general solution to  \eqref{eq:shift_eq} is given by
	$$s_G(t) = \frac{1}{m_0}\pa{e^{-\CC t}-e^{-\pa{\CC+m_0\KK}t}}\bm{m}_1+e^{-\pa{\CC+m_0\KK}t}s_G(0)$$

%	To solve  \eqref{eq:shift_eq} we notice that a solution to the associated homogenous equation is given by
%	$$s_H(t)=e^{-\pa{\CC+m_0\KK}t}s_H(0)$$
%	and using the method of variation of parameters, the general solution to \eqref{eq:shift_eq} is
%%	$$s(t) = e^{-\pa{C+m_0K}t}\alpha(t)$$
%%	and as such
%%	$$Ke^{-Ct}\bm{m}_1-\pa{C+m_0K}s(t) = e^{-\pa{C+m_0K}t}\alpha^\prime(t)-\pa{C+m_0K}s(t)$$
%%	or equivalently
%%	$$\alpha^\prime(t) = e^{\pa{C+m_0K}t}Ke^{-Ct}\bm{m}_1$$
%%	We conclude that the solution to \eqref{eq:shift_eq} is given by
%	\begin{equation}\nonumber
%		\begin{gathered}
%	%		s(t) =  e^{-\pa{C+m_0K}t}s(0) +  e^{-\pa{C+m_0K}t}\int_0^t e^{\pa{C+m_0K}s}Ke^{-Cs}\bm{m}_1ds\\
%			s_G(t)=e^{-\pa{\CC+m_0\KK}t}s_G(0) + \int_0^t e^{\pa{\CC+m_0\KK}(s-t)}\KK e^{-\CC s}\bm{m}_1ds.
%		\end{gathered}
%	\end{equation}
	As was discussed in \cite{arnold2014}, the conditions on the pair $\pa{\CC+m_0\KK,\DD}$ guarantee the existence of a unique smooth solution to the Fokker-Planck equation 
	\begin{equation}\label{eq:FP_intermittent}
		\partial_t\widetilde{\rho}(x,t)=\div_x\rpa{\pa{\CC+ m_0\KK}x\widetilde{\rho}(x,t)+ \DD\nabla_x\widetilde{\rho}(x,t)} 
	\end{equation}
	as long as the initial datum is in $L^1_+\pa{\R^d}\cap L^p\pa{\R^d}$ for some $p\in (1,\infty]$. We conclude that 
		\begin{equation}\label{eq:sol}
		\rho_G(x,t) = \widetilde{\rho}(x-s_G(t),t)=f_{FP}^{\CC+m_0\KK,\DD}\pa{x-s_G(t),t}
	\end{equation}
	is a candidate for a smooth and strongly decaying solution to \eqref{eq:MKE}.  By choosing $s_G(0)=0$, i.e. considering the function  
	$$\rho(x,t) =f_{FP}^{\CC+m_0\KK,\DD}\pa{x-s(t),t}$$
	where 
		\begin{equation}\label{eq:general_s}
		s(t)=\frac{1}{m_0}\pa{e^{-\CC t}-e^{-\pa{\CC+m_0\KK}t}}\bm{m}_1,
		\end{equation}
 we see that
	$$\widetilde{\rho}(x,0)=\rho(x+s(0),0)=\rho(x,0)=\rho_0(x)\in L^1_+\pa{\R^d}\cap L^p\pa{\R^d}$$
	and consequently have managed to find a potential solution. Plugging this function back in \eqref{eq:MKE} completes the proof.
\end{proof}

\begin{remark}\label{rem: explicit solution of FPE}
The verification that \eqref{eq:sol} is indeed a solution to \eqref{eq:MKE} relies on the explicit form of $f_{FP}^{\CC+m_0\KK,\DD}$. It was shown in \cite{arnold2014}, the solution of the Fokker Planck equation \eqref{eq: FPE} when the pair $\pa{\A,\B}$ is admissible is given explicitly by
\begin{equation}\label{eq:FP_form}
f_{FP}^{\A,\B}(x,t)=\frac{1}{\pa{2\pi \det \mathcal{Q}(t)}^\frac{d}{2}}\int_{\R^d}e^{-\frac{1}{2}(x-e^{-\mathcal{A}t}y)^T \mathcal{Q}(t)^{-1}(x-e^{-\mathcal{A}t}y)}\rho_0(y)\,dy,
\end{equation}
where the convolution kernel $\mathcal{Q}(t)$ is given by
\begin{equation}\nonumber
\mathcal{Q}(t):=2 \int_{0}^t e^{-\mathcal{A}s}\mathcal{B} e^{-\mathcal{A}^T s}ds.	
\end{equation}
We leave the proof to the Appendix as it is a straightforward calculation.
\end{remark}

With an explicit expression of the solution to the McKean-Vlasov equation, one that is intimately connected to a Fokker-Planck equation, we can start considering the question of convergence to equilibrium.

\section{About the convergence to equilibrium}\label{sec:convergence}

The main result shown in the previous section - namely the fact that a solution to our McKean-Vlasov equation \eqref{eq:MKE} is nothing but a ``transported'' solution to an appropriate Fokker-Planck equation - allows us to utilise recent sharp results about convergence to equilibrium in the setting of Fokker-Planck equations, in particular those achieved in the work of Monmarch{\'e} \cite{MoH19}, to conclude explicit and quantitative convergence to equilibrium in \eqref{eq:MKE}
% to not only conclude cases where we have convergence to equilibrium for our equation - but also ones where we have stabilisation around a certain profile function. 

We begin our section by quoting the main convergence to equilibrium result from \cite{MoH19}:

\begin{theorem}\label{thm:monmarche}
	Consider the Fokker-Planck equation 
	\begin{equation}\nonumber
		\partial_t f(x,t)=\div\rpa{\A x f(x,t)+  \B\nabla f(x,t)}
	\end{equation}
	where the pair $\pa{\A,\B}$ is admissible. Let $f(\cdot,t)$ be the solution to the above equation with a unit mass initial datum $f_0\in L^1_+\pa{\R^d}\cap L^p\pa{\R^d}$ for some $p\in (1,\infty]$ and let 	
	\begin{equation}\nonumber
		\eta: = \min\br{\Re \lambda\;|\; \lambda\text{ is an eigenvalue of }\A}.
	\end{equation}
	Then, there exists an explicit $c>0$ that only depends on $\A$ and $\B$ such that
	\begin{equation}\nonumber
		\HH\pa{f(t)|f_\infty} \leq c \HH\pa{f_0|f_\infty}\pa{1+t^{2n}}e^{-2\eta t}
	\end{equation}
	where $n$ is the maximal defect of all eigenvalues associated to $\eta$ and where
	\begin{equation}\nonumber
		f_\infty(x):=\frac{1}{\pa{2\pi \det \mathcal{Q}}^{\frac{d}{2}}}e^{-\frac{1}{2}x^T\mathcal{Q}x},
	\end{equation}
	with $\mathcal{Q}$ being the unique positive definite solution to 
	\begin{equation}\nonumber
		2\B=\A\mathcal{Q}+\mathcal{Q} \A^T.
	\end{equation} 
\end{theorem}

The above theorem and Theorem \ref{thm:solution_to_MKE} provide us with the following result: 
%which guarantees that under weaker conditions than those of Theorem \ref{thm:main} we can still get a long-time behaviour \textit{profile} to our solution:
\begin{theorem}\label{thm:profile}
	Consider the McKean-Vlasov equation \eqref{eq:MKE} with drift, interaction, and diffusion matrices $\CC$, $\KK$, and $\DD$ respectively. Assume in addition that the pair $\pa{\CC+\KK,\DD}$ is admissible and that the equation admits a unique solution. Let $\rho(\cdot,t)$ be the solution to the equation with initial unit mass datum $\rho_0\in L^1_+\pa{\R^d}\cap L^p\pa{\R^d}$ for some $p\in (1,\infty]$. Then, following on the notations for $s(t)$, $\rho_\infty$, $\mu$, and $n_1$ from Theorem \ref{thm:main} we have that
	\begin{equation}\label{eq:profile_convergence}
		\HH\pa{\rho(t)|\rho_\infty\pa{\cdot-s(t)}}\leq c \HH\pa{\rho_0|\rho_\infty}\pa{1+t^{2n_1}}e^{-2\mu t}
	\end{equation}
	for an explicit $c>0$ that only depends on $\CC$, $\KK$, and $\DD$. 
\end{theorem}

\begin{proof}
	Following on Theorem \ref{thm:solution_to_MKE} we have that
	$$\HH\pa{\rho(t)|\rho_\infty\pa{\cdot-s(t)}} = \HH\pa{f_{FP}^{\CC+\KK,\DD}\pa{\cdot-s(t),t},|\rho_\infty\pa{\cdot-s(t)}}$$
	$$=\HH\pa{f_{FP}^{\CC+\KK,\DD}\pa{t}|\rho_\infty}\leq c \HH\pa{f_0|\rho_\infty}\pa{1+t^{2n_1}}e^{-2\mu t},$$
	where we have used the fact that a time translation in the spatial variable doesn't change the value of the relative entropy and Theorem \ref{thm:monmarche}.
\end{proof}

Theorem \ref{thm:profile} assures us that the solution to \eqref{eq:MKE} will ``stabilise'', in a sense, as long as $\pa{\CC+\KK,\DD}$ (or $\pa{\CC+m_0\KK,\DD}$ to be more general) is an admissible pair. The profile function, however, given by $\rho_\infty\pa{\cdot-s(t)}$ is not time independent in general. It is not surprising, however, that if $s(t)$ converges as time goes to infinity then a true equilibrium (which is by definition stationary) emerges. 

We present the following technical lemma to show that under the conditions of Theorem \ref{thm:main} this is always the case.

\begin{lemma}\label{lem:convergence_of_s}
	Let $\A$ and $\B$ be $d\times d$ matrices such that $\pa{\A+\B}$ is positively stable.
	Define the operator
	$$\xi(t):=e^{-\A t} - e^{-\pa{\A+\B}t}.$$
	Then $\xi(t)$ has a limit as $t$ goes to infinity if and only if $\A$ is almost-positively stable and in that case
	$$\lim_{t\to\infty} \xi(t) = P_{\ker \A}.$$
	 Moreover, if we denote by 
		\begin{equation}\nonumber
		\alpha: = \min\br{\Re \lambda\;|\; \lambda\text{ is an eigenvalue of }\A+\B},
	\end{equation}
	and
	\begin{equation}\nonumber
		\beta:=\begin{cases}
			\min\br{\Re \lambda\;|\; \lambda\not=0\text{ is an eigenvalue of }\A}, & \A\not= 0,\\
			0, & \A=0.
		\end{cases}
	\end{equation}
	Then for any $x\in \R^d$
	\begin{equation}\label{eq:convergence_of_s}
		\begin{aligned}
				\abs{\xi(t)x-P_{\ker \A}x} \leq & c_{\A}\pa{1+t^{n_2}}e^{-\beta t}\abs{\pa{\II-P_{\ker \A}}x}\\
				&+ c_{\A+\B}\pa{1+t^{n_1}}e^{-\alpha t}\abs{x},
		\end{aligned}
	\end{equation}
	where $c_{\A}>0$ and $c_{\A+\B}>0$ are explicit constants that depend only on $\A$ and $\B$, and  $n_1$ and $n_2$ are the maximal defect of all eigenvalues associated to $\alpha$ and $\beta$ respectively ($n_2$ is defined to be zero when $\A=0$). 
\end{lemma}

\begin{proof}
	To show the desired result, we start by considering single $r\times r$ Jordan block matrix of the form 
	$$\mathfrak{J} = \begin{pmatrix}
		\lambda & 1 & 0 & \dots & 0 & 0 \\
		0 & \lambda &  1 & \dots & 0 & 0 \\
		\vdots & \vdots & \vdots & \vdots & \vdots & \vdots \\
		0 & 0 &  0 & \dots & \lambda & 1\\
		0 & 0 &  0 & \dots & 0 & \lambda  \\
	\end{pmatrix}$$
 A well known result from the theory of ODEs states that we can find an explicit constant $c_{\mathfrak{J}}$ such that\footnote{the interested reader may find a proof to this statement in the Proof of Lemma 4.7 in \cite{arnold2018}.}
 $$\norm{e^{-\mathfrak{J}t}} \leq c_{\mathfrak{J}}\pa{1+t^{r}}e^{-\lambda t}.$$
 The above implies that for any matrix $\mathcal{C}$ with a Jordan form 
 \begin{equation}\nonumber
 	\mathcal{J}=\begin{pmatrix}
 		\mathcal{J}_1 & \bm{0} & \dots & \bm{0} \\
 			\vdots & \vdots & \vdots & \vdots  \\
 		\bm{0} & \bm{0} & \dots & \mathcal{J}_l	
 	\end{pmatrix}
 \end{equation}
 we have that
 \begin{equation}\label{eq:general_positively_stable_decay}
 	\norm{e^{-\mathcal{C}t}} \leq c_{S} \norm{e^{-\mathcal{J}t}} \leq c_{S}\max_{i=1,\dots, l}c_{\mathcal{J}_i}\pa{1+t^{r_i}}e^{-\lambda_i t}
 \end{equation}
where $c_S$ is a constant that depends on the similarity matrix taking $\mathcal{C}$ to $\mathcal{J}$.\\
 We can thus conclude that, with the notation of the lemma, if $\A+\B$ is a positively stable matrix then
\begin{equation}\label{eq:norm_for_A+B}
	\norm{e^{-\pa{\A+\B}t}}\leq c_{\A+\B}\pa{1+t^{n_1}}e^{-\alpha t}.
\end{equation}
Next we notice that the definition of $P_{\ker \A}$, \eqref{eq:proj}, for an almost-positively stable matrix imply that 
$$e^{-\A t} - P_{\ker\A} = e^{-\A t}\pa{\II-P_{\ker\A}}.$$
Consequently, if we consider the $\A-$invariant space $V:=\mathrm{span}\pa{\II-P_{\ker\A}}$ on which $\A$ is positively stable (see Remark \ref{rem:on_projection}), we find that \eqref{eq:general_positively_stable_decay} implies that for any $x\in\R^d$
\begin{equation}\label{eq:norm_for_A_on_x}
	\begin{aligned}
		&\norm{e^{-\A t} x- P_{\ker\A}x} = \norm{e^{-\A t}\pa{\II-P_{\ker \A}}x}\\
		&\leq \norm{e^{-\A t}\vert_{V}}\abs{\pa{\II-P_{\ker\A}}x}\leq c_{\A}\pa{1+t^{n_2}}e^{-\beta t}\abs{\pa{\II-P_{\ker\A}}x}.
	\end{aligned}
\end{equation}
Since
$$	\abs{\xi(t)x-P_{\ker \A}x} \leq \abs{e^{-\A t}x - P_{\ker \A}x} + \abs{e^{-\pa{\A+\B}t}x}$$
the desired inequality follows directly from \eqref{eq:norm_for_A+B} and \eqref{eq:norm_for_A_on_x}.
\end{proof}

\begin{remark}\label{rem:on_s(t)}
	Lemma \ref{lem:convergence_of_s} and the conditions for Theorem \ref{thm:main} guarantee that $s(t)$, which is given by \eqref{eq:def_of_shift}, will converge to $s_\infty=P_{\ker \CC}\bm{m}_1$ as time goes to infinity. Moreover, using the notations of Theorem \ref{thm:main}
	\begin{equation}\label{eq:rate_of_convergence_of_s}
		\begin{aligned}
		\abs{s(t)-s_\infty} \leq  & c\pa{1+t^{n_2}}e^{-\nu t}\abs{\pa{\II-P_{\ker \CC}}\bm{m}_1}\\
		&+ c\pa{1+t^{n_1}}e^{-\mu t}\abs{\bm{m}_1},
		\end{aligned}
	\end{equation}
	for any appropriate explicit constant $c>0$.
	
	It is worth to mention that $s(t)$ may converge under less restrictive conditions than the ones imposed in Theorem \ref{thm:main}. For instance, if $\bm{m}_1=0$ then $s(t)=0$ for all time regardless of $\CC$ and $\KK$.
\end{remark}

We are now ready to prove our main theorem.

\begin{proof}[Proof of Theorem \ref{thm:main}]
	We start by noticing that 
	\begin{equation}\label{eq:equilibrium_proof_I}
	\begin{gathered}
		\HH\pa{\rho(t)|\rho_\infty\pa{\cdot-s_{\infty}}} = \int_{\R^d}\rho(x,t)\log\pa{\frac{\rho(x,t)}{\rho_\infty\pa{x-s_\infty}}}dx\\
		= \HH\pa{\rho(t)|\rho_\infty\pa{\cdot-s(t)}}+\int_{\R^d}\rho(x,t)\log\pa{\frac{\rho_\infty(x-s(t))}{\rho_\infty\pa{x-s_\infty}}}dx.
	\end{gathered}
	\end{equation}
	Since
	$$\frac{\rho_\infty(x-s(t))}{\rho_\infty\pa{x-s_\infty}} = e^{-\frac{1}{2}\rpa{\pa{x-s(t)}^T\K \pa{x-s(t)} - \pa{x-s_\infty}^T\K \pa{x-s_\infty} }}$$
	$$=e^{\pa{s(t)-s_\infty}^T\K x }e^{-\frac{1}{2}\rpa{s(t)^T \K s(t) - s_\infty^T \K s_\infty}}=e^{\pa{s(t)-s_\infty}^T\K x}e^{-\frac{1}{2}\pa{s(t)-s_\infty}^T\K \pa{s(t)+s_\infty}}, $$
where we have used the fact that for any symmetric matrix $\A$
$$\pa{x-y}^T \A \pa{x+y} = x^T \A x - y^T \A y,$$
we conclude that 
\begin{equation}\label{eq:equilibrium_proof_II}
	\begin{aligned}
	&\HH\pa{\rho(t)|\rho_\infty\pa{\cdot-s_{\infty}}} =\HH\pa{\rho(t)|\rho_\infty\pa{\cdot-s(t)}}\\
	&+\int_{\R^d}\pa{s(t)-s_\infty}^T\K x \rho(x,t)dx\\
	&- \frac{1}{2}\pa{s(t)-s_\infty}^T\K \pa{s(t)+s_\infty}\int_{\R^d}\rho\pa{x,t}dx\\
	&=\HH\pa{\rho(t)|\rho_\infty\pa{\cdot-s(t)}}+\pa{s(t)-s_\infty}^T\K e^{-\CC t}\bm{m}_1 \\
	& - \frac{1}{2}\pa{s(t)-s_\infty}^T\K \pa{s(t)+s_\infty}
	\end{aligned}
\end{equation}
where have used the conservation of (unit) mass and \eqref{eq:evolution_of_first_moment}.

Consequently, using the fact that $\K$ is positive definite and as such there exists $\kappa>0$ such that $\K\leq \kappa \II$, we find that
\begin{equation}\label{eq:equilibrium_proof_III}
	\begin{aligned}
		&\HH\pa{\rho(t)|\rho_\infty\pa{\cdot-s_{\infty}}} \leq \HH\pa{\rho(t)|\rho_\infty\pa{\cdot-s(t)}} \\
	% 		&+ \frac{\kappa}{2}\abs{s(t)-s_\infty}\pa{\abs{P_{\ker \CC}\bm{m}_1}+\norm{e^{-\CC t}-P_{\ker \CC}}\abs{\bm{m}_1}}\\	
		&+ \frac{\kappa}{2}\abs{s(t)-s_\infty}\pa{\abs{P_{\ker \CC}\bm{m}_1}+\abs{\pa{e^{-\CC t}-P_{\ker \CC}}\bm{m}_1}}\\
		&+\frac{\kappa}{2}\abs{s(t)-s_\infty}^2+ \kappa\abs{s_\infty}\abs{s(t)-s_\infty}.
	\end{aligned}
\end{equation}
Using Theorem \ref{thm:profile},inequality \eqref{eq:rate_of_convergence_of_s}, the facts that $s_\infty=P_{\ker \CC}\bm{m}_1$ and   
$$\abs{\pa{e^{-\CC t}-P_{\ker \CC}}\bm{m}_1} \leq c_{\CC}\pa{1+t^{n_2}}e^{-\nu t}\abs{\pa{\II-P_{\ker\CC}}\bm{m}_1}$$
we find that
%	\begin{equation}\nonumber
%	\begin{aligned}
%		\abs{s(t)-P_{\ker \CC}\bm{m}_1} \leq  & c\pa{1+t^{n_1+n_2+\delta_{\mu,\nu}}}e^{-\min\pa{\mu,\nu}t}\abs{\bm{m}_1}\\
%		&+c\pa{1+t^{n_1}}e^{-\mu t}\abs{P_{\ker \CC}\bm{m}_1},
%	\end{aligned}
%\end{equation}
%we find that
\begin{equation}\nonumber %\label{eq:equilibrium_proof_IV}
	\begin{aligned}
		&\HH\pa{\rho(t)|\rho_\infty\pa{\cdot-s_{\infty}}} \leq c \HH\pa{\rho_0|\rho_\infty}\pa{1+t^{2n_1}}e^{-2\mu t}\\
		&+c\rpa{\pa{1+t^{n_2}}e^{-\nu t}\abs{\pa{\II-P_{\ker \CC}}\bm{m}_1}+\pa{1+t^{n_1}}e^{-\mu t}\abs{\bm{m}_1}}\\
		&\Big[\pa{1+\pa{1+t^{n_1}}e^{-\mu t}}\abs{P_{\ker \CC}\bm{m}_1} + \pa{1+t^{n_2}}e^{-\nu t}\abs{\pa{\II-P_{\ker\CC}}\bm{m}_1}\Big]\\
	%	&+\pa{\pa{1+t^{n_2}}e^{-\nu t}+\pa{1+t^{n_1+n_2+\delta_{\mu,\nu}}}e^{-\min\pa{\mu,\nu}t}}\abs{\bm{m}_1}\Big],\\
%		&+ \frac{\kappa}{2}\abs{s(t)-s_\infty}\pa{\abs{P_{\ker \CC}\bm{m}_1}+\norm{e^{-\CC t}-P_{\ker \CC}}\abs{\bm{m}_1}}\\
%		&+\frac{\kappa}{2}\abs{s(t)-s_\infty}^2+ \kappa\abs{s_\infty}\abs{s(t)-s_\infty}.
	\end{aligned}
\end{equation}
for an appropriate constant $c>0$, which is the desired result. The proof is now complete. 
\end{proof}

\section{Final remarks}\label{sec:final}

While our main estimate, inequality \eqref{eq:explicit_convergence}, seems complicated, it gives us an easily computable and explicit rate of convergence to equilibrium. It is clear from the proof of our theorems, that various additional conditions on the drift, interaction, and diffusion matrices will allow us to refine our proofs and get better rates of convergence. For example,  in the case $\KK=0$  we find that $s(t)=s_\infty=0$ (as expected). Consequently, the rate of convergence to equilibrium, according to Theorem \ref{thm:profile}, is given by $\pa{1+t^{2n_1}}e^{-2\mu t}$, which is a better rate than the one we'll get from \eqref{eq:explicit_convergence}.

%Instead, we have elected to try and provide the most ``cost effective'' choice of conditions on these matrices and a concrete convergence result. 

We would like to emphasise that the method we presented here can deal with many degenerate cases, such as the case $\CC=0$. In that case, as long as the pair $\pa{\KK,\DD}$ is admissible, we find that 
$$\HH\pa{\rho(t)|\rho_\infty\pa{\cdot-s_\infty}}\leq c \HH\pa{\rho_0|\rho_\infty}\pa{1+t^{2n_1}}e^{-2\mu t}+c\pa{1+t^{n_1}}e^{-\mu t}\abs{\bm{m}_1}.$$
This showcase the importance of the interaction term, governed by $\KK$, as in this instance a part of it takes the role of a drift which allows us to attain an equilibrium.

Another point we would like to mention is that the identity
\begin{equation}\nonumber
	\begin{aligned}
		&\HH\pa{\rho(t)|\rho_\infty\pa{\cdot-s_{\infty}}} =\HH\pa{\rho(t)|\rho_\infty\pa{\cdot-s(t)}}\\
		&+\pa{s(t)-s_\infty}^T\K e^{-\CC t}\bm{m}_1 
		 - \frac{1}{2}\pa{s(t)-s_\infty}^T\K \pa{s(t)+s_\infty},
	\end{aligned}
\end{equation}	
which was shown in the proof of our main theorem, Theorem \ref{thm:main}, shows that the optimal rate of convergence to equilibrium is determined by the profile convergence, $\HH\pa{\rho_0|\rho_\infty}\pa{1+t^{2n_1}}e^{-2\mu t}$, \textit{and} the optimal rate of convergence of $s(t)$ to $s_\infty$. As 
$$s(t)=\pa{e^{-\CC t}-e^{-\pa{\CC+\KK}t}}\bm{m}_1,$$
we immediately see the significant impact of $\CC$ and $\KK$, as well as the interplay between them, on the convergence. As an example we notice that since
\begin{equation}\nonumber
	\begin{aligned}
		\abs{s(t)-s_\infty} \leq  & c\pa{1+t^{n_2}}e^{-\nu t}\abs{\pa{\II-P_{\ker \CC}}\bm{m}_1}\\
		&+ c\pa{1+t^{n_1}}e^{-\mu t}\abs{\bm{m}_1},
	\end{aligned}
\end{equation}
when $\CC$ is not positively stable (but is almost positively stable), the above will dominate the convergence to equilibrium.

%It is important to point out that the presence of the interaction term, governed by $\KK$, has a huge impact on the \hd{existence of an unique equilibrium and the} convergence rate. Indeed, this term is what generates the time shift in the spatial variable, $s(t)$, whose convergence to $s_\infty=P_{\ker \CC}\bm{m}_1$ \textit{dominates} the rate of convergence we provided, as can be seen clearly in the proof of Theorem \ref{thm:main}. 

We would like to note that the methodology presented in this worked relied heavily on the fact that the interaction kernel was of a linear form, i.e. that $k(x,y)=\KK(x-y)$. This allowed us to recast our McKean-Vlasov equation as a standard Fokker-Planck equation with a ``transportation'' term. It is not clear to us if a more complicated interaction kernel will allow for similar manipulation which will result in more general Fokker-Planck equation but it is a direction we intend to explore. 

Lastly, we would like to mention that it would be interesting to investigate whether or not the result attained in this work can be achieved via a process of mean field limit from the particle system associated to the McKean-Vlasov equation \eqref{eq:MKE}.

	\section*{Acknowledgement}
	The authors would like to thank the reviewers of this work for their helpful comments and suggestions which have improved upon the work.

\appendix

\section{The verification of our solution}
In this short appendix we will show that the function
	$$\rho(x,t) =f_{FP}^{\CC+m_0\KK,\DD}\pa{x-s(t),t},$$
with $f_{FP}^{\A,\B}$ defined as in \eqref{eq:FP_form} and $s(t)$ defined in \eqref{eq:def_of_shift}, is a solution to our McKean-Vlasov equation \eqref{eq:MKE}.

  Since
	$$\frac{1}{\pa{2\pi \det\mathcal{Q}(t)}^{\frac{d}{2}}}\int_{\R^d}xe^{-\frac{1}{2}(x-e^{-\mathcal{A}t}y)^T \mathcal{Q}(t)^{-1}(x-e^{-\mathcal{A}t}y)}dx=e^{-\A t}y$$
	%$$=\frac{e^{-\A t}y}{(2\pi)^\frac{d}{2}\sqrt{\det \mathcal{Q}(t)}}\int_{\R^d}e^{-\frac{1}{2}z^T \mathcal{Q}(t)^{-1}z}dz + \frac{1}{(2\pi)^\frac{d}{2}\sqrt{\det \mathcal{Q}(t)}}\int_{\R^d}ze^{-\frac{1}{2}z^T \mathcal{Q}(t)^{-1}z}dz=e^{-\A t}y$$
	we have that
	$$\KK\int_{\R^d}xf_{FP}^{\CC+m_0\KK,\DD}\pa{x-s(t),t}dx=\KK s(t)\int_{\R^d}f_{FP}^{\CC+m_0\KK,\DD}\pa{x,t}dx$$
	$$+\KK\int_{\R^d}xf_{FP}^{\CC+m_0\KK,\DD}\pa{x,t}dx= \KK s(t) \int_{\R^d}\rho_0(y)dy + \KK\int_{\R^d}e^{-\pa{\CC+m_0\KK}t}y\rho_0(y)dy$$
	$$=m_0\KK s(t)  +\KK e^{-\pa{\CC+m_0\KK}t} \bm{m}_1=\KK\pa{e^{-\CC t}-e^{-\pa{\CC+m_0\KK}t}}\bm{m}_1$$
	$$+ \KK e^{-\pa{\CC+m_0\KK}t} \bm{m}_1 = \KK e^{-\CC t}\bm{m}_1.$$
	Consequently
	$$\KK \int_{\R^d}\pa{x-y}f_{FP}^{\CC+m_0\KK,\DD}\pa{y-s(t),t}dy=m_0\KK x-\KK e^{-\CC t}\bm{m}_1$$
	and we find that $\rho(x,t)$ satisfies
	$$\partial_t\rho(x,t) = \partial_t f_{FP}^{\CC+m_0\KK,\DD}(y(x,t),t) - \nabla_y f_{FP}^{\CC+m_0\KK,\DD}(y(x,t),t)\cdot \frac{ds}{dt}(t)$$
	$$\div_y\pa{\pa{\CC+m_0\KK} y(x,t) \rho(x,t)+\DD\nabla_x \rho(x,t)}-\pa{\KK e^{-\CC t}\bm{m}_1-\pa{\CC+m_0\KK}s(t)}\nabla_y f_{FP}^{\CC+m_0\KK,\DD}(y(x,t),t)$$
	$$=\div_y\pa{\pa{\pa{\CC+m_0\KK} y(x,t)-\pa{\KK e^{-\CC t}\bm{m}_1-\pa{\CC+m_0\KK}s(t)}} \rho(x,t)+\DD\nabla_x \rho(x,t)}$$
	$$=\div_y\pa{\pa{\pa{\CC+m_0\KK} \pa{y(x,t)+s(t)}-\KK e^{-\CC t}\bm{m}_1} \rho(x,t)+\DD\nabla_x \rho(x,t)}$$
	$$=\div\pa{\pa{\pa{\CC+m_0\KK} x-\KK e^{-\CC t}\bm{m}_1} \rho(x,t)+\DD\nabla_x \rho(x,t)}$$
	$$=\div\pa{\CC x \rho(x,t)+\pa{m_0\KK x-\KK e^{-\CC t}\bm{m}_1} \rho(x,t)+\DD\nabla_x \rho(x,t)}$$
	$$\div\pa{\CC x \rho(x,t)+ \pa{\int_{\R^d} \KK(x-y)\rho(y,t)dy}\rho(x,t)+ \DD\nabla \rho(x,t)},$$
	where we have used the notation $y(x,t)=x-s(t)$. This shows the desired result.
\bibliographystyle{plain}

\end{document}